\documentclass[10pt]{article}
\usepackage{graphics}
\usepackage{latexsym}
\usepackage{times}
\usepackage[dvipdfm]{graphicx}
\catcode`@=11
\renewcommand{\@begintheorem}[2]{\it \trivlist            
      \item[\hskip \labelsep{\bf #1\ #2{\rm :}}]}         
\renewcommand{\@opargbegintheorem}[3]{\it \trivlist       
      \item[\hskip \labelsep{\bf #1\ #2\ {\rm (#3)\/:}}]}
\def\@sect#1#2#3#4#5#6[#7]#8{\ifnum #2>\c@secnumdepth
     \def\@svsec{}\else 
     \refstepcounter{#1}\edef\@svsec{\csname the#1\endcsname{.}\hskip 1em }\fi
     \@tempskipa #5\relax
      \ifdim \@tempskipa>\z@ 
        \begingroup #6\relax
          \@hangfrom{\hskip #3\relax\@svsec}{\interlinepenalty \@M #8\par}
        \endgroup
       \csname #1mark\endcsname{#7}\addcontentsline
         {toc}{#1}{\ifnum #2>\c@secnumdepth \else
                      \protect\numberline{\csname the#1\endcsname}\fi
                    #7}\else
        \def\@svsechd{#6\hskip #3\@svsec #8\csname #1mark\endcsname
                      {#7}\addcontentsline
                           {toc}{#1}{\ifnum #2>\c@secnumdepth \else
                             \protect\numberline{\csname the#1\endcsname}\fi
                       #7}}\fi
     \@xsect{#5}}

\@addtoreset{equation}{section}
\newcommand{\Delete}[1]{}

\usepackage{amsmath,amssymb,amsthm}

\theoremstyle{plain}
\newtheorem{Thm}{Theorem}[section]
\newtheorem{Lem}[Thm]{Lemma}
\newtheorem{Prop}[Thm]{Proposition}

\newcommand{\bD}{\ensuremath{\mathbb{D}}}

\newcommand{\bH}{\ensuremath{\mathbb{H}}}
\newcommand{\bI}{\ensuremath{\mathbb{I}}}

\newcommand{\bO}{\ensuremath{\mathbb{O}}}

\newcommand{\bT}{\ensuremath{\mathbb{T}}}

\newcommand{\bZ}{\ensuremath{\mathbb{Z}}}

\newcommand{\cF}{\ensuremath{\mathcal{F}}}

\newcommand{\cS}{\ensuremath{\mathcal{S}}}

\title{The competition numbers of regular polyhedra}

\author{
\begin{tabular}{c}
{\sc Yoshio SANO}\\
\\
Research Institute for Mathematical Sciences, \\
Kyoto University, Kyoto 606-8502, Japan. \\ 
{\tt sano@kurims.kyoto-u.ac.jp}
\end{tabular}
}

\date{May 2009}

\begin{document}

\maketitle

\begin{abstract}
The notion of a competition graph was introduced by J. E. Cohen in 1968. 
The {\it competition graph} $C(D)$ of a digraph $D$ 
is a (simple undirected) graph which 
has the same vertex set as $D$ 
and has an edge between two distinct vertices $x$ and $y$ 
if and only if 
there exists a vertex $v$ in $D$ such that 
$(x,v)$ and $(y,v)$ are arcs of $D$. 
For any graph $G$, $G$ together with sufficiently many new isolated vertices 
is the competition graph of some acyclic digraph. 
In 1978, F. S. Roberts defined 
the {\it competition number} $k(G)$ of a graph $G$ 
as the minimum number of such isolated vertices. 
In general, it is hard to compute the competition number 
$k(G)$ for a graph $G$ and it has been one of 
important research problems in the study of 
competition graphs 
to characterize a graph by computing its competition number. 

It is well known that there exist $5$ kinds of regular polyhedra 
in the three dimensional Euclidean space: 
a tetrahedron, a hexahedron, an octahedron, 
a dodecahedron, and an icosahedron.  
We regard a polyhedron as a graph. 
The competition numbers of 
a tetrahedron, a hexahedron, an octahedron, and a dodecahedron 
are easily computed by using known results on competition numbers. 
In this paper, we focus on an icosahedron and 
give the exact value of the competition number 
of an icosahedron. 
\end{abstract}

\noindent
{\bf Keywords:} 
competition graph; 
competition number; 
edge clique cover; 
regular polyhedra; 
icosahedron


\newpage
\section{Introduction}

Throughout this paper, all graphs $G$ are simple and undirected. 
The notion of a competition graph was introduced by J. E. Cohen \cite{Cohen1} 
in connection with a problem in ecology 
(see also \cite{Cohen2}).
The {\it competition graph} $C(D)$ of a digraph $D$ is a graph
which has the same vertex set 
as $D$ and has an edge between two distinct vertices 
$x$ and $y$ if and only if 
there exists a vertex $v$ in $D$ 
such that $(x,v)$ and $(y,v)$ are arcs of $D$. 
For any graph $G$, $G$ together 
with sufficiently many 
isolated vertices is the competition graph 
of an acyclic digraph. 
From this observation, 
F. S. Roberts \cite{MR0504054} defined the {\it competition number} $k(G)$ of
a graph $G$ to be the minimum number $k$ such that $G$ together with 
$k$ isolated vertices is the competition graph of an acyclic digraph: 
\begin{equation}
k(G):= \min \{k \in \bZ_{\geq 0} \mid G \cup I_k = C(D) 
\text{ for some acyclic digraph } D \}, 
\end{equation}
where $I_k$ denotes a set of $k$ isolated vertices. 

For a digraph $D$, 
an ordering $v_1, v_2, \ldots, v_n$ of the vertices of $D$
is called an {\it acyclic ordering} of $D$
if $(v_i,v_j) \in A(D)$ implies $i<j$.
It is well known that a digraph $D$ is acyclic if and only if
there exists an acyclic ordering of $D$.

A subset $S \subseteq V(G)$ of the vertex set of a graph $G$ 
is called a {\it clique} of $G$ if the subgraph $G[S]$ 
of $G$ induced by $S$ is a complete graph. 
For a clique $S$ of a graph $G$ and an edge $e$ of $G$,
we say {\it $e$ is covered by $S$} 
if both of the endpoints of $e$ are contained in $S$.
An {\it edge clique cover} of a graph $G$
is a family of cliques such that
each edge of $G$ is covered by some clique in the family 
(see \cite{MR770871} for applications of edge clique covers). 
The {\it edge clique cover number} $\theta_E(G)$ of a graph $G$ 
is the minimum size of an edge clique cover of $G$. 
A {\it vertex clique cover} of a graph $G$ 
is a family of cliques such that 
each vertex of $G$ is contained in some clique in the family. 
The {\it vertex clique cover number} $\theta_V(G)$ of a graph $G$
is the minimum size of a vertex clique cover of $G$.

R. D.  Dutton and R. C. Brigham \cite{MR712930} 
characterized the competition graph of an acyclic digraph 
in terms of an edge clique cover. 
(F. S. Roberts and J. E. Steif \cite{MR712932} 
characterized the competition graphs of loopless digraphs. 
J. R. Lundgren and J. S. Maybee \cite{MR712931} 
gave a characterization of graphs whose competition number 
is at most $m$.) 

However, R. J. Opsut \cite{MR679638} showed that 
the problem of determining 
whether a graph is the competition graph of an acyclic digraph 
or not is NP-complete. 
Therefore it follows that the computation of the 
competition number of a graph is an NP-hard problem, 
and thus it does not seem to be easy in general to compute 
$k(G)$ for an arbitrary graph $G$ 
(see \cite{kimsu}, \cite{MR1833332}, \cite{kr} 
for graphs whose competition numbers are known). 
It has been one of important research problems in the study of 
competition graphs to characterize a graph by computing 
its competition number. 

F. S. Roberts showed the following: 
\begin{Thm}[Roberts \cite{MR0504054}]\label{thm:chordal}
If $G$ is a chordal graph without isolated vertices, then 
$k(G)=1$. 
\end{Thm}

\begin{Thm}[Roberts \cite{MR0504054}]\label{thm:tri-free}
If $G$ is a connected triangle-free graph with $|V(G)|>1$, then 
\begin{equation}
k(G) = |E(G)|-|V(G)|+2. 
\end{equation}
\end{Thm}

From the above theorems, we can compute the competition numbers of 
many kinds of graphs: for example, 
the competition number of a complete bipartite graph $K_{n,n}$ 
is equal to $n^2-2n+2$. 
S. -R. Kim and Y. Sano gave a formula for the competition numbers of 
complete tripartite graphs $K_{n,n,n}$, which we cannot obtain 
by the above theorems. 
\begin{Thm}[Kim and Sano \cite{MR2467965}]\label{thm:KimSano}
For $n \geq 2$, 
\begin{equation}
k(K_{n,n,n}) = n^2-3n+4. 
\end{equation}
\end{Thm}

R. J. Opsut gave the following two lower bounds 
for the competition number of a graph. 

\begin{Thm}[Opsut \cite{MR679638}, Proposition 5]\label{thm:OpsutBdE}
For any graph $G$, 
\begin{equation}\label{eq:OpsutBdE}
k(G) \geq \theta_E(G)-|V(G)|+2. 
\end{equation}
\end{Thm}

\begin{Thm}[Opsut \cite{MR679638}, Proposition 7]\label{thm:OpsutBdV}
For any graph $G$, 
\begin{equation}\label{eq:OpsutBdV}
k(G) \geq \min \{ \theta_V(N_{G}(v)) \mid v \in V(G) \}, 
\end{equation}
where $N_G(v):=\{u \in V(G) \mid uv \in E(G) \}$ 
is the open neighborhood of a vertex $v$ in the graph $G$. 
\end{Thm}

Recently, Y. Sano gave a generalization of 
the above two lower bounds as follows: 

Let $G$ be a graph and $F \subseteq E(G)$ be a subset of the edge set of $G$. 
An {\it edge clique cover} of $F$ in $G$ is a family of cliques of $G$ 
such that each edge in $F$ is covered by some clique in the family. 
We define the {\it edge clique cover number} 
$\theta_E(F;G)$ of $F \subseteq E(G)$ in $G$ as 
the minimum size of an edge clique cover of $F$ in $G$: 
\begin{equation}
\theta_E(F;G) := \min \{|\cS| \mid \cS \text{ is an edge clique cover of } 
F \text{ in } G \}. 
\end{equation}
Note that 
the edge clique cover number 
$\theta_E(E(G); G)$ 
of $E(G)$ in a graph $G$ is equal to 
the edge clique cover number $\theta_E(G)$ of the graph $G$. 

Let $U \subseteq V(G)$ be a subset of the vertex set 
of a graph $G$. We define 
\begin{eqnarray}
N_G[U] &:=& \{v \in V(G) \mid v 
\text{ is adjacent to a vertex in } U \} \cup U, \\
E_G[U] &:=& \{e \in E(G) \mid e \text{ has an endpoint in } U \}. 
\end{eqnarray}
We denote by the same symbol $N_G[U]$ the subgraph of $G$ induced by $N_G[U]$. 
Note that $E_G[U]$ is contained in the edge set of the subgraph $N_G[U]$. 
We denote by ${V \choose m}$ the set of all $m$-subsets of a set $V$. 

\begin{Thm}[Sano \cite{Sano}]\label{thm:Sano}
Let $G=(V,E)$ be a graph. 
Let $m$ be an integer such that $1 \leq m \leq |V|$. 
Then 
\begin{equation}\label{eq:lem:main}
k(G) \geq 
\min_{U \in {V \choose m}}
\theta_E(E_{G}[U]; N_{G}[U]) -m+1. 
\end{equation}
\end{Thm}

In this paper, we compute the competition numbers 
of regular polyhedra. 
One of our motivation of this paper is coming from the following: 
G. Chen {\it et al.} 
mentioned that 
``the icosahedron -- a claw-free graph -- 
has competition number at least three." in Conclusion of their paper 
\cite{MR1761707} 
with referring Theorem \ref{thm:OpsutBdV}. 
But they didn't give the exact value of 
the competition number of an icosahedron. 
So we compute it! 
This paper gives the exact value of the competition number of an icosahedron.

\section{Regular Polyhedra}

It is well known that 
there are $5$ regular polyhedra (or Platonic solids) 
in the three dimensional Euclidean space: 
a tetrahedron $\bT$, a hexahedron $\bH$, an octahedron $\bO$, 
a dodecahedron $\bD$, and an icosahedron $\bI$. 
We regard these polyhedra as graphs (see Figure \ref{fig:polyhedra}). 
In this subsection, we give the exact values of the competition numbers 
of these $5$ regular polyhedra. 

The competition numbers of 
a tetrahedron, a hexahedron, an octahedron, and a dodecahedron 
are easily computed by using known results on competition numbers as follows: 

\begin{figure}
\begin{center}
\includegraphics[width=292pt]{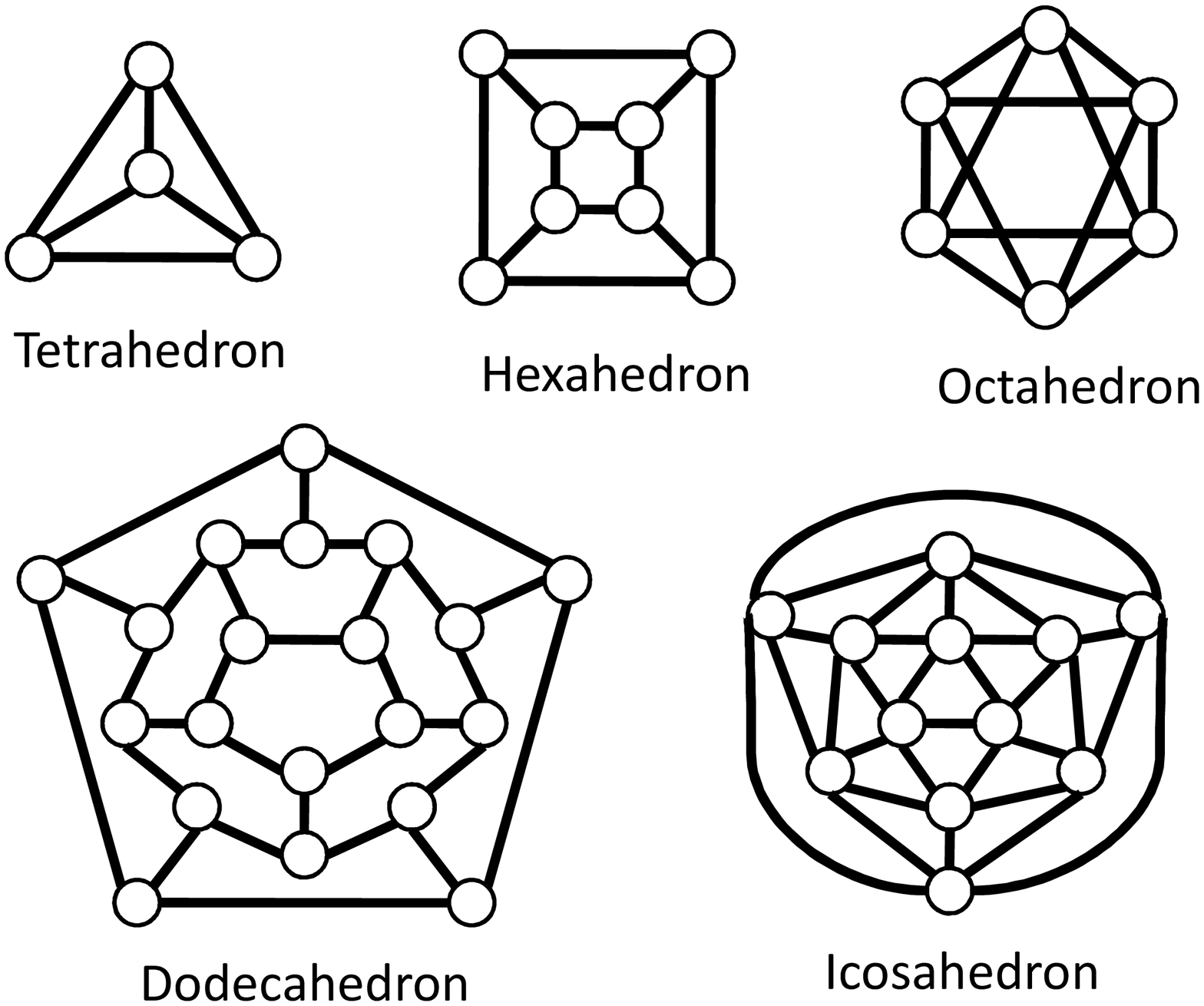} 
\caption{The $5$ regular polyhedra (Platonic solids)}
\label{fig:polyhedra}
\end{center}
\end{figure}

\begin{Thm}
Let $\bT$ be a tetrahedron. Then 
$k(\bT)=1$. 
\end{Thm}

\begin{proof}
Since $\bT \cong K_4$ 
is a chordal graph without isolated vertices, 
we have $k(\bT)=1$ by Theorem \ref{thm:chordal}. 
\end{proof}

\begin{Thm}
Let $\bH$ be a hexahedron. Then 
$k(\bH)=6$. 
\end{Thm}

\begin{proof}
Since $\bH$ is a triangle-free connected graph with $|V(\bH)|>1$, 
we have $k(\bH)=12-8+2=6$ by Theorem \ref{thm:tri-free}. 
\end{proof}

\begin{Thm}
Let $\bO$ be an octahedron. Then 
$k(\bO)=2$. 
\end{Thm}

\begin{proof}
Since $\bO \cong K_{2,2,2}$ 
is a complete tripartite graph, 
we have $k(\bO)=2^2 - 3 \cdot 2 +4=2$ 
by Theorem \ref{thm:KimSano}. 
\end{proof}

\begin{Thm}
Let $\bD$ be a dodecahedron. Then 
$k(\bD)=12$. 
\end{Thm}

\begin{proof}
Since $\bD$ is a triangle-free connected graph with $|V(\bD)|$, 
we have $k(\bD)=30-20+2=12$ by Theorem \ref{thm:tri-free}. 
\end{proof}

Now, we focus on an icosahedron. 
The competition number of an icosahedron is 
given as follows: 

\begin{Thm}\label{thm:main}
Let $\bI$ be an icosahedron. Then 
$k(\bI)=4$. 
\end{Thm}

A proof of Theorem \ref{thm:main} will be given in the next section.

\section{Proof of Theorem \ref{thm:main}}

In this section, we give a proof of Theorem \ref{thm:main}. 

First, we show the lower bound part of the proof. 

\begin{Lem}\label{lem:LB}
$k(\bI) \geq 4$. 
\end{Lem}

\begin{proof}
By Theorem \ref{thm:Sano} with $m=3$, 
we have 
\[
k(\bI) \geq \min_{U \in {V(\bI) \choose 3}} 
\theta_E(E_{\bI}[U]; N_{\bI}[U]) -2. 
\]
There are the following $4$ cases 
for the subgraph $\bI[U]$ of an icosahedron $\bI$ induced by 
$U \in {V(\bI) \choose 3}$ 
(see Figure \ref{fig:neighbor}): \\
(i) If $\bI[U]$ is a triangle $K_3$, 
then $\theta_E(E_{\bI}[U]; N_{\bI}[U])=6$. \\
(ii) If $\bI[U]$ is a path $P_3$ with $3$ vertices, 
then $\theta_E(E_{\bI}[U]; N_{\bI}[U])=6$. \\
(iii) If $\bI[U]=K_2 \cup I_1$, 
then $\theta_E(E_{\bI}[U]; N_{\bI}[U])=7$. \\
(iv) If $\bI[U]=I_3$, 
then $\theta_E(E_{\bI}[U]; N_{\bI}[U])=9$. \\
\begin{figure}
\begin{center}
\includegraphics[width=338pt]{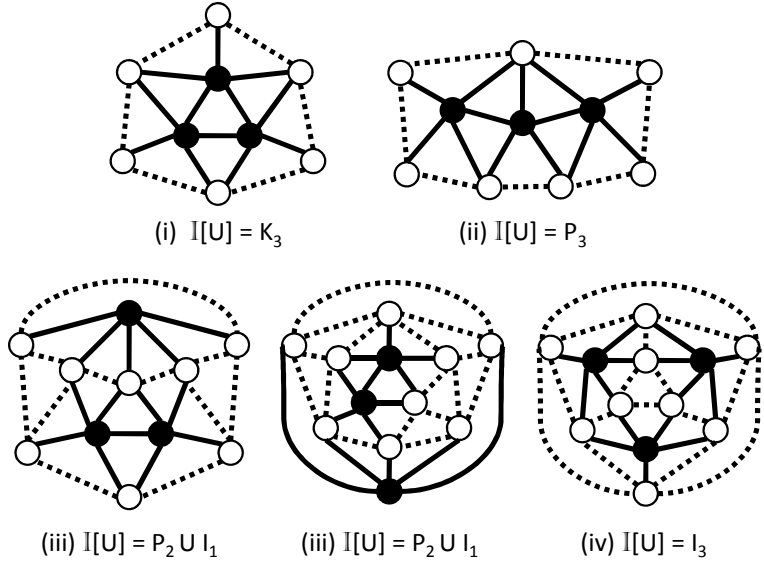} 
\caption{The set $E_{\bI}[U]$ of edges in the subgraph $N_{\bI}[U]$}
\label{fig:neighbor}
\end{center}
\end{figure}
Thus it holds that 
$\min_{U \in {V(\bI) \choose 3}} \theta_E(E_{\bI}[U]; N_{\bI}[U])=6$. 
Hence we have $k(\bI) \geq 4$. 
\end{proof}

The edge clique cover number of an icosahedron is given as follows: 

\begin{Prop}
$\theta_E(\bI)=12$. 
\end{Prop}

\begin{proof}
Let $\cF=\{S_1, \ldots, S_r\}$ be an edge clique cover 
of $\bI$ with the minimum size $r:=\theta_E(\bI)$. 
For any vertex $v \in V(\bI)$, 
we have 
\begin{equation}\label{eq:001}
|\{S_i \in \cF \mid v \in S_i \}| \geq 3 
\end{equation}
since $\theta_V(N_{\bI}(v))=\theta_V(C_5)=3$. 
Let 
\[
N:=|\{(v, S_i) \in V(\bI) \times \cF \mid v \in S_i \}|. 
\]
Since $|V(\bI)|=12$ and (\ref{eq:001}), 
we have $N \geq 12 \times 3$. 
Since the maximum size of a clique of $\bI$ is $3$, 
we have $N \leq 3 \times r$. 
Therefore it holds that $r \geq 12$. 
\end{proof}

Second, we show the upper bound part of the proof. 

\begin{Lem}\label{lem:UB}
$k(\bI) \leq 4$. 
\end{Lem}

\begin{figure}
\begin{center}
\includegraphics[width=334pt]{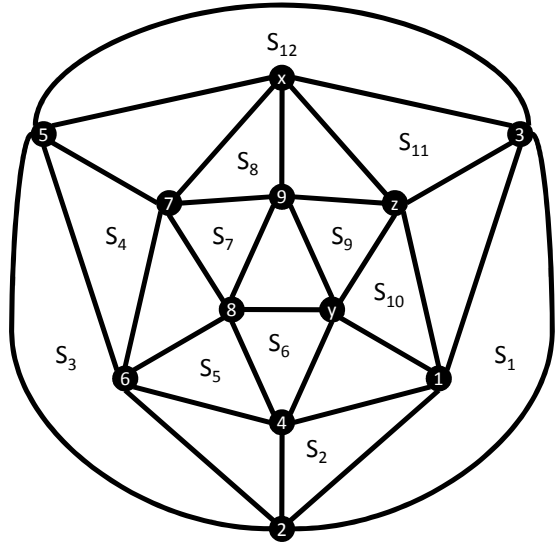} 
\caption{An edge clique cover of an icosahedron $\bI$}
\label{fig:ordering}
\end{center}
\end{figure}

\begin{proof}
Let 
\[
V(\bI) := \{v_1, v_2, v_3, 
v_4, v_5, v_6, v_7, v_8, v_9, v_x, v_y, v_z \}. 
\]
and suppose that 
the adjacencies between two vertices are given as 
Figure \ref{fig:ordering}. 
Let 
\begin{eqnarray*}
&& 
S_1 := \{ v_1, v_2, v_3\}, \quad
S_2 := \{ v_1, v_2, v_4\}, \quad
S_3 := \{ v_2, v_5, v_6\}, \quad
S_4 := \{ v_5, v_6, v_7\}, \\ &&
S_5 := \{ v_4, v_6, v_8\}, \quad
S_6 := \{ v_4, v_8, v_y\}, \quad
S_7 := \{ v_7, v_8, v_9\}, \quad
S_8 := \{ v_7, v_9, v_x\}, \\ &&
S_9    := \{ v_9, v_y, v_z\}, \quad
S_{10} := \{ v_1, v_y, v_z\}, \quad
S_{11} := \{ v_3, v_x, v_z\}, \quad
S_{12} := \{ v_3, v_5, v_x\}. 
\end{eqnarray*}
Then the family 
$\{S_1, S_2, \ldots, S_{12}\}$ 
is an edge clique cover of an icosahedron $\bI$ 
(see Figure \ref{fig:ordering}). 

Now, we define a digraph $D$ by the following: 
\[
V(D):=V(\bI) \cup \{ a,b,c,d \}, 
\]
\begin{eqnarray*}
S_1 = \{ v_1, v_2, v_3\} \to a, &&
S_2 = \{ v_1, v_2, v_4\} \to v_3, \\
S_3 = \{ v_2, v_5, v_6\} \to v_1, &&
S_4 = \{ v_5, v_6, v_7\} \to v_2, \\
S_5 = \{ v_4, v_6, v_8\} \to v_5, &&
S_6 = \{ v_4, v_8, v_y\} \to v_6, \\
S_7 = \{ v_7, v_8, v_9\} \to v_4, &&
S_8 = \{ v_7, v_9, v_x\} \to v_8, \\
S_9    = \{ v_9, v_y, v_z\} \to v_7, &&
S_{10} = \{ v_1, v_y, v_z\} \to b, \\
S_{11} = \{ v_3, v_x, v_z\} \to c, &&
S_{12} = \{ v_3, v_5, v_x\} \to d, 
\end{eqnarray*}
where $a, b, c, d$ are new vertices, 
and ``$S \to v$" means that we make an arc from each vertex in $S$ to 
the vertex $v$. 
Then the digraph $D$ is acyclic and 
$C(D)=\bI \cup \{a,b,c,d\}$. 
Hence we have $k(\bI) \leq 4$. 
\end{proof}

Now, we complete the proof of Theorem \ref{thm:main}. 

\begin{proof}[Proof of Theorem \ref{thm:main}]
It follows from Lemmas \ref{lem:LB} and \ref{lem:UB}. 
\end{proof}

\section{Closing Remark}

In this paper, we gave the exact values of the competition numbers 
of regular polyhedra, especially gave the competition number of 
an icosahedron. 
It would be interesting to compute the competition numbers 
of some triangulations of a sphere.

\section*{Acknowledgment}

The author was supported by JSPS Research Fellowships 
for Young Scientists. 
The author was also supported partly by Global COE program 
``Fostering Top Leaders in Mathematics".



\begin{thebibliography}{99}%


\bibitem{MR1761707}
{G. Chen, M. S. Jacobson, A. E. K{\'e}zdy, 
J. Lehel, E. R. Scheinerman, and C. Wang}:
{Clique covering the edges of a locally cobipartite graph}, 
{\it Discrete Math.}, {\bf 219} (2000) 17--26.


\bibitem{Cohen1}
{J. E. Cohen}:
{Interval graphs and food webs: a finding and a problem},
{\it Document 17696-PR}, RAND Corporation,
Santa Monica, CA (1968).

\bibitem{Cohen2}
{J. E. Cohen}:
{\it Food webs and Niche space},
Princeton University Press, Princeton, NJ (1978).

\bibitem{MR712930}
{R. D. Dutton and R. C. Brigham}:
{A characterization of competition graphs},
{\it Discrete Appl. Math.}, {\bf 6} (1983) 315--317.

\bibitem{kimsu}
{S. -R. Kim}:
The Competition Number and Its Variants,
in {\it Quo Vadis, Graph Theory?},
(J. Gimbel, J. W. Kennedy, and L. V. Quintas, eds.),
{\em Annals of Discrete Mathematics} {\bf 55},
North Holland B. V., Amsterdam, the Netherlands (1993) 313--326.

\bibitem{MR1833332}
{S. -R. Kim}:
{On competition graphs and competition numbers}, (in Korean), 
{\it Commun. Korean Math. Soc.}, {\bf 16} (2001) 1--24.

\bibitem{kr}
{S. -R. Kim and F. S. Roberts}:
Competition numbers of graphs with a small number of triangles,
{\it Discrete Appl. Math.}, {\bf 78} (1997) 153--162.

\bibitem{MR2467965}
{S. -R. Kim and Y. Sano}:
{The competition numbers of complete tripartite graphs}, 
{\it Discrete Appl. Math.}, {\bf 156} (2008) 3522--3524.

\bibitem{MR712931}
{J. R. Lundgren and J. S. Maybee}:
{A characterization of graphs of competition number {$m$}}, 
{\it Discrete Appl. Math.}, {\bf 6} (1983) 319--322.

\bibitem{MR679638}
{R. J. Opsut}:
{On the computation of the competition number of a graph},
{\it SIAM J. Algebraic Discrete Methods}, {\bf 3} (1982)
420--428.


\bibitem{MR0504054}
{F. S. Roberts}:
{Food webs, competition graphs, and the boxicity of ecological phase space},
{\it Theory and applications of graphs (Proc. Internat. Conf.,
Western Mich. Univ., Kalamazoo, Mich., 1976)}, {\bf } (1978)
477--490.

\bibitem{MR770871}
{F. S. Roberts}:
{Applications of edge coverings by cliques}, 
{\it Discrete Appl. Math.}, {\bf 10} (1985) 93--109.

\bibitem{MR712932}
{F. S. Roberts and J. E. Steif}:
{A characterization of competition graphs of arbitrary 
digraphs}, 
{\it Discrete Appl. Math.}, {\bf 6} (1983) 323--326.


\bibitem{Sano}
{Y. Sano}:
{A generalization of Opsut's lower bounds 
for the competition number of a graph },
{\it preprint} {\bf RIMS-1663}, March 2009.
(http://www.kurims.kyoto-u.ac.jp/preprint/file/RIMS1663.pdf)


\end{thebibliography}
\end{document}